\documentclass{paper}       

\usepackage{graphicx,subfigure,paralist}
\usepackage{latexsym,amsfonts,amssymb,amsmath,enumerate,framed,epsfig,amsthm}
\usepackage{color}
\usepackage{authblk}

\newcommand{\und}{\wedge}

\newcommand{\vx}{{\mathbf x}}
\newcommand{\vz}{{\mathbf z}}
\newcommand{\vw}{{\mathbf w}}
\newcommand{\vn}{{\boldsymbol{\nu}}}

\newcommand{\Bildeinbinden}[2] {\setlength{\epsfxsize}{#1}\epsfbox{#2}}

\newtheorem{theorem}{Theorem}[section]

\newtheorem{lemma}{Lemma}[section]

 \numberwithin{equation}{section}

\allowdisplaybreaks

\begin{document}

\author{Steffen Goebbels\footnote{Niederrhein University of Applied Sciences, Faculty of Electrical Engineering
and Computer Science, Institute for Pattern
Recognition, D-47805 Krefeld, Germany, steffen.goebbels@hsnr.de}}

\title{On Sharpness of Error Bounds for Multivariate Neural Network
Approximation}

\maketitle
\begin{center}
Third uploaded version, 23.11.2020. This preprint has been accepted by
Ricerche di Mathmatica.
\end{center}

\begin{abstract}
Single hidden layer feedforward neural networks can represent multivariate
functions that are sums of ridge functions. These ridge functions are defined
via an activation function and customizable weights.
The paper deals with best non-linear approximation by such sums of ridge
functions. Error bounds are presented in terms of moduli of smoothness. But the
main focus is on proving that the bounds are best possible. To this end,
counterexamples are constructed with a non-linear, quantitative extension of the
uniform boundedness principle. They show sharpness with respect to
Lipschitz classes for the logistic activation function and for certain piecewise
polynomial activation functions. The paper is based on univariate
results in (Goebbels, S.: On sharpness of error bounds
for univariate approximation by single hidden layer feedforward neural networks. Results Math., accepted for
publication; http://arxiv.org/abs/1811.05199).
\end{abstract}

\\[1em]

\section{Introduction} 
A feedforward neural network with an activation function
$\sigma:\mathbb{R}\to\mathbb{R}$, $d$ input nodes, one output node, and one
hidden layer of $n$ neurons implements a multivariate real-valued function $g :\mathbb{R}^d\to\mathbb{R}$ of type 
\begin{alignat}{1}
g(\vx) \in{\cal M}_n &:= {\cal M}_{n,\sigma}:=\left\{  \sum_{k=1}^n a_k
\sigma(\vw_k \cdot \vx+c_k) : a_k, c_k\in\mathbb{R}, \vw_k\in\mathbb{R}^d 
\right\},\label{network}
\end{alignat}
see Figure \ref{fignet}.
For vectors $\vw_k =(w_{k,1},\dots, w_{k,d})$ and
$\vx=(x_1,\dots,x_d)$,
$$\vw_k\cdot\vx =
\sum_{j=1}^d w_{k,j} x_{j}$$ 
is the standard inner product of $\vw_k,\,\vx\in \mathbb{R}^d$. Summands $\sigma(\vw_k
\cdot \vx+c_k)$ are ridge functions. They are constant on hyperplanes $\vw_k
\cdot \vx= c$, $c\in\mathbb{R}$.
\begin{figure}
\begin{center}
\Bildeinbinden{0.5\columnwidth}{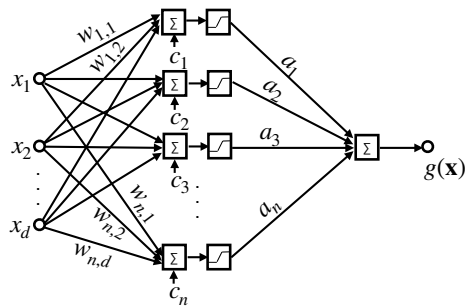}
\end{center}
\caption{One hidden layer neural network realizing $\sum_{k=1}^n a_k
\sigma(\vw_k \cdot \vx+c_k)$\label{fignet}}
\end{figure}

For non-constant bounded monotonically increasing and continuous activation
functions $\sigma$ the universal approximation property holds, see the paper of
Funahashi \cite{Funahashi89}: 
Given a continuous function $f$ on a compact set then for
each $\varepsilon>0$ there exists $n\in\mathbb{N}$ and a function $g_n\in {\cal
M}_n$ such that the sup-norm of $f-g_n$ is less than $\varepsilon$. Cybenko
showed this property in \cite{Cybenko89} for a different class of continuous
activation functions that do not have to be monotone.
Leshno et al.~proved for 
continuous activation functions $\sigma$ in \cite{Leshno93}
(cf.~\cite{Chen95}) that the universal approximation property is equivalent
to $\sigma$ being not an algebraic polynomial. Continuity is not a
necessary prerequisite for the universal approximation property, see
\cite{Jones90}.

We discuss error bounds for best approximation by functions of
${\cal M}_n$ in terms of moduli of smoothness for an arbitrary number $d$ of
input nodes in Section \ref{secdirect}. These results are quantitative
extensions of the qualitative universal approximation property. They introduce
convergence orders that depend on the smoothness of functions to be
approximated.

Many papers deal with the univariate case $d=1$. In \cite{Chen93}, Debao proved
an estimate against a first oder modulus for general sigmoid activation
functions. An overview of other estimates against first order moduli is given in
doctoral thesis \cite{Costarelli14}, cf.\ \cite{COSTARELLI2013101}.
Under additional assumptions on activation functions, estimates against higher
order moduli are possible. For example, one can easily extend the first order
estimate of Ritter for approximation with ``nearly exponential'' activation functions in
\cite{Ritter99} to higher moduli, see \cite{Goebbels20}. 
Similar results can be obtained for activation
functions that are arbitrarily often differentiable on some open interval such that
they are not an algebraic polynomial on that interval, see \cite[Theorem 6.8,
p.~176]{pinkus_1999} in combination with \cite{Goebbels20}.

With respect to the general multivariate case, Barron applied Fourier methods
in \cite{Barron93} to establish a convergence rate for a
certain class of smooth functions in the $L^2$-norm. 
Approximation errors for multi-dimensional bell shaped activation functions were
estimated by first order moduli of smoothness or related Lipschitz classes by
Anastassiou (e.g.~\cite{Anastassiou00}) 
and Costarelli and Spigler (see e.g.~\cite{Costarelli13}
including a literature overview). However, discussed neural network spaces
differ from (\ref{network}). They do not consist of linear combinations of ridge
functions.
A special network with four layers is
introduced in \cite{LIN20146031} to obtain a Jackson estimate in terms of a
first order modulus of smoothness.

Maiorov and Ratsby establish an upper bound for functions in
Sobolev spaces based on pseudo-dimension in \cite[Theorem 2]{Ratsaby99}.
Pseudo-dimension is an upper bound of 
the Vapnik-Chervonenkis dimension (VC dimension) that will also be used in this
paper to obtain lower bounds.

With respect to neural network spaces (\ref{network}) of ridge functions, we
apply results of Pinkus
\cite{pinkus_1999} and Maiorov and Meir \cite{Maiorov99}
to
obtain error bounds for a
large class of activation functions either based on K-functional techniques
or on known estimates for best approximation with multivariate polynomials in
Section \ref{secdirect}.
Both $L^p$- and sup-norms are considered.

In Section \ref{secount}, we prove for the logistic
activation function that counterexamples
$f_\alpha$ exist for all $\alpha>0$ such that sup-norm as well as
$L^p$-norm bounds are in $O(n^{-\alpha})$ but the error of best
approximation is not in $O(n^{-\beta})$ for $\beta > \alpha$.
This result is a multivariate extension of univariate counterexamples
($d=1$, one single input node, sup-norm) in \cite{Goebbels20}. 
A similar result is shown for piecewise polynomial activation functions with
respect to an $L^2$-norm bound.

In fact, the 
non-linear
variant of a quantitative uniform boundedness principle in \cite{Goebbels20}
can be applied to construct univariate and multivariate counterexamples.
This principle is based on theorems of Dickmeis, Nessel and van Wickeren,
cf.~\cite{DiNeWi2}, that can be used to analyze error bounds of linear
approximation processes. Its application, both in a linear and in the given
non-linear context, requires the construction of a resonance sequence. To this
end, a known result \cite{Bartlett1996} on the 
VC dimension of networks with logistic activation is used. Theorem \ref{thVC}
in Section \ref{secount} is formulated as a general means to derive
discussed counterexamples from VC dimension estimates. Also, \cite{Maiorov99}
already provides sequences of counterexamples that can be condensed to a single
counterexample with the uniform boundedness principle.

There are some published attempts to show sharpness of error bounds for neural
network approximation in terms of moduli of smoothness based on inverse
theorems. Inverse and equivalence theorems estimate the values of moduli of
smoothness by approximation rates. For example, they determine membership to
certain Lipschitz classes from known approximation errors. 
However, the letter \cite{Goebbels19} proves 
that the inverse theorem for neural network approximation in \cite{Wang10} as
well as the inverse theorems in some related papers are wrong. Smoothness is one
feature that favors high approximation rates. But in this non-linear
situation, other features (e.g. the ``nearly exponential'' property or
similarity to certain derivatives of the activation function,
cf.~\cite{Kurkova98}) also contribute to convergence rates.
Such features cannot be sufficiently measured by moduli of smoothness,
cf.~sequence of counterexamples in \cite{Goebbels19}.
This is the motivation to work with counterexamples instead of inverse or 
equivalence theorems in Section \ref{secount}.

\section{Notation and Direct Estimates}\label{secdirect}
Let $\Omega\subset\mathbb{R}^d$ be an open set.
By $X^p(\Omega):=L^p(\Omega)$ with norm 
$$\|f\|_{L^p(\Omega)}:=\root p \of
{\int_\Omega |f(\vx)|^p d\vx}$$ for $1\leq p<\infty$ and
$X^\infty(\Omega):=C(\overline{\Omega})$ with sup-norm
$\|f\|_{C(\overline{\Omega})}:=\sup\{|f(\vx)| : \vx\in\overline{\Omega}\}$ 
we denote the usual Banach spaces.

For a multi-index $\alpha=(\alpha_1,\dots,\alpha_d)\in\mathbb{N}_0^d$ with
non-negative integer components, let $|\alpha|:= \sum_{j=1}^d \alpha_j$
be its sum of components. We write
$\vx^\alpha := \prod_{j=1}^d x_j^{\alpha_j}$.
With $P_k$ we denote the set of
multivariate polynomials with degree at most $k$, i.e., each polynomial in $P_k$
is a linear combination of homogeneous polynomials of degree
$l\in\{0,\dots,k\}$. To this end, let $$H_l:= \left\{ f:\mathbb{R}^d
\to\mathbb{R} :
f(\vx)= \sum_{\alpha \in \mathbb{N}_0^d,\,|\alpha|=l} c_\alpha
\vx^\alpha\right\}
$$
be the space of homogeneous polynomials of degree $l$.

The set of all univariate polynomials with degree at most $k$ is denoted by
$\Pi_k$, i.e., $\Pi_k=P_k$ for $d=1$.
Let $$s:=\dim H_k = \binom{d+k-1}{k} \leq (k+1)^{d-1}.$$
To obtain the upper estimate, we choose exponents
$\alpha_1,\dots,\alpha_{d-1}$ in $\vx^\alpha$ independently from
the set $\{0,\dots,k\}$.
If the sum of these exponents does not exceed $k$ then
$\alpha_{d}=k-\sum_{j=1}^{d-1} \alpha_j$.
Otherwise, we have counted a polynomial with degree greater than $k$. Thus, the
estimate only is a coarse upper bound.

Multivariate polynomials can be represented by
univariate polynomials, cf.~\cite[p.~164]{pinkus_1999}: For a given degree
$k\in\mathbb{N}$ there exist $s\leq (k+1)^{d-1}$ vectors $\vw_1,\dots,\vw_s
\in\mathbb{R}^d$
such that
\begin{equation}
 P_k=\left\{\sum_{j=1}^s p_j(\vw_j\cdot\vx) : p_j\in \Pi_k
 \right\}.\label{poldicht}
\end{equation}
We use the result \cite[p.~176]{pinkus_1999}, cf.\ \cite{Ito}:
Let $\sigma:\mathbb{R}\to\mathbb{R}$ be arbitrarily often differentiable
on some open interval $I\subset\mathbb{R}$, i.e.~$\sigma\in C^\infty(I)$, 
and let $\sigma$ be no algebraic polynomial on that interval. Then univariate polynomials of degree at most $k$ can be uniformly
approximated arbitrarily well on compact sets by choosing parameters
$a_j,b_j,c_j\in\mathbb{R}$ in $$\sum_{j=1}^{k+1} a_j\sigma(b_j x + c_j).$$ 
Thus due to (\ref{poldicht}),
also multivariate polynomials of degree at most $k$ can be approximated by
functions of ${\cal M}_{s(k+1)}$ arbitrarily well on compact sets, i.e., in
the sup-norm
\begin{equation}
P_k \subset \overline{{\cal M}_{s(k+1)}}.\label{dense}
\end{equation}
Theorem 3.1 in \cite{Kurkova98} even describes a more general class of
multivariate functions that can be approximated arbitrarily well like
polynomials.

There holds following lemma from
\cite[Proposition 4]{Ito} that extends (\ref{dense}) to simultaneous
approximation.
\begin{lemma}\label{LaIto}
Let $\sigma :\mathbb{R}\to\mathbb{R}$ be arbitrarily often differentiable on an
open interval around the origin with $\sigma^{(i)}(0)\neq 0$,
$i\in\mathbb{N}_0$.
Then for any polynomial $\pi\in P_k$ of degree at most $k$, any compact set
$I\subset\mathbb{R}^d$, and each $\varepsilon>0$ there exists a sufficiently
often differentiable function $g\in {\cal M}_{s(k+1)}$ such that
simultaneously for all $\alpha\in\mathbb{N}_0^d$, $|\alpha|\leq k$, 
$$ \left\| \frac{\partial^{|\alpha|}}{\partial x_1^{\alpha_1} \dots \partial
x_d^{\alpha_d}} \left(\pi(\vx) - g(\vx)\right)\right\|_{C(I)} <
\varepsilon.$$
\end{lemma}
The requirement that derivatives at zero must not be zero can be replaced by the
requirement that $\sigma$ is no algebraic polynomial on the open interval, see
\cite{Ito}.

With $n$ summands of the activation function, polynomials of degree $k$ and
their derivatives can be simultaneously approximated arbitrarily well for such
values of $k$ that fulfill 
$n\geq (k+1)^d \text{, i.e.~} k \leq \root d \of n -1$, because
$(k+1)^d = (k+1)^{d-1} (k+1) \geq s(k+1)$.
Especially, polynomials of degree at most
\begin{equation}
k := \lfloor \root d \of n \rfloor -1\label{umrechnung}
\end{equation}
can be approximated arbitrarily well.

Let
$f\in X^p(\Omega)$,
$\vn\in\mathbb{R}^d$, $r\in\mathbb{N}_0$ und $t\in\mathbb{R}$.
The $r$th radial difference (with direction $\vn$) is given via
\begin{equation*}
 \Delta_\vn^r f(\vx) := \sum_{j=0}^r (-1)^{r-j} \binom{r}{j} f(\vx+j\vn)
\end{equation*}
(if defined).
Thus, $\Delta_\vn^r f=\Delta_\vn^{r-1}\Delta_\vn^1 f
=\Delta_\vn^1 \Delta_\vn^{r-1} f$.
Let
$$
 \Omega(\vn) := \{\vx\in\Omega : \vx+t \vn\in\Omega,\,
  0\leq t\leq 1\}.
$$
Then the $r$th radial modulus of smoothness of a function
$f\in L^p(\Omega)$,
$1\leq p<\infty$, or $f\in C(\overline{\Omega})$ is defined via
\begin{alignat*}{1}
\omega_r(f,\delta)_{p,\Omega}&:=\sup\{ \|\Delta_\vn^r f
   \|_{X^p(\Omega(r\vn))} : \vn\in\mathbb{R}^d,\, |\vn|\leq\delta\}.
\end{alignat*}

Our aim is to discuss errors $E$ of best approximation. For $S\subset
X^p(\Omega)$ and $f\in X^p(\Omega)$ let
$$ E(S, f)_{p,\Omega} := \inf \{ \|f-g\|_{X^p(\Omega)} : g\in S\}.
$$
Thus, $E(S, f)_{p,\Omega}$ is the distance between $f$ and $S$.

As an application of a multivariate equivalence theorem between K-functional
and moduli of smoothness, an estimate for best polynomial approximation is
proved on Lipschitz graph domains (LG-domains) 
in \cite[Corollary 4, p.~139]{JohnenScherer}. For the definition of not
necessarily bounded LG-domains, see \cite[p.~66]{Adams}. For bounded domains,
the LG property is equivalent to a Lipschitz boundary. Especially, later
discussed bounded $d$-dimensional open intervals like $(0,1)^d$ and the unit
ball $\{\vx\in\mathbb{R}^d : |\vx|<1\}$ are examples for LG-domains.

Let $\Omega$ be a bounded LG-domain in $\mathbb{R}^n$ and $1\leq p\leq\infty$,
then 
\begin{equation}
 E(P_k, f)_{p,\Omega} \leq C_r\omega_r\left(f,
\frac{1}{k}\right)_{p,\Omega}\label{eqPol}
\end{equation}
with a constant $C_r$ that is independent of $f$ and $k$, see 
\cite{JohnenScherer}.

\begin{theorem}[Arbitrarily Often Differentiable Functions]\label{theorem1}
Let $\sigma:\mathbb{R}\to\mathbb{R}$ be arbitrarily often differentiable
on some open interval in $\mathbb{R}$, and let $\sigma$ be no
algebraic polynomial on that interval, $f\in X^p(\Omega)$ for an LG-domain
$\Omega\in\mathbb{R}^d$, $1\leq p\leq \infty$, and $r\in\mathbb{N}$.
For $n\geq 4^d$ there exists a constant
$C$ that is independent of $f$ and $k$
such that
\begin{alignat*}{1}
&E({\cal M}_n, f)_{p,\Omega} \leq C \omega_r\left(f, \frac{1}{\root d \of n} \right)_{p,\Omega}.
\end{alignat*}
\end{theorem}

\begin{proof}
We combine (\ref{dense}) and (\ref{umrechnung}) with (\ref{eqPol}) to get
\begin{alignat*}{1}
\lefteqn{
E({\cal M}_n, f)_{p,\Omega}  \leq C_r
\omega_r\left(f,\frac{1}{ \lfloor \root d \of n \rfloor -1}\right)_{p,\Omega}
\leq C_r
\omega_r\left(f,\frac{1}{ \root d \of n -2} \right)_{p,\Omega}}\\
&\leq C_r
\omega_r\left(f, \frac{1}{ \root d \of n -\frac{\root d \of n}{2}}
\right)_{p,\Omega}
= C_r
\omega_r\left(f, \frac{2}{\root d \of n} \right)_{p,\Omega}
\leq \underbrace{C_r 2^r}_{C} \omega_r\left(f, \frac{1}{\root d \of n}
\right)_{p,\Omega}.
\end{alignat*}
\end{proof}
By using an error bound for best polynomial approximation we are not able to
consider advantages of non-linear approximation. However, we will see in the
next section that non-linear neural network approximation does not really
perform better than polynomial approximation in the worst case.

Most activation functions, that are not piecewise polynomials, fulfill the
requirements of Theorem \ref{theorem1}. For example, it provides an error bound
for approximation with the sigmoid activation function based on inverse
tangent 
$$\sigma(x) =\frac{1}{2}+\frac{1}{\pi} \arctan(x),$$ the logistic function 
$$\sigma(x)
=\frac{1}{1+e^{-x}}=\frac{1}{2}\left(1+\tanh\left(\frac{x}{2}\right)\right),$$ 
and ''Exponential Linear Unit'' (ELU) activation function 
$$ \sigma(x) = \left\{ \begin{array}{cl} \alpha(e^x-1),& x<0\\
x ,& x\geq 0\end{array}\right. $$
for $\alpha\neq 0$.

A direct bound for simultaneous approximation of a function and its partial
derivatives in the sup-norm can be obtained similarly based on a corresponding
estimate for simultaneous approximation by polynomials using a Jackson estimate
from \cite{Bagby02}:
\begin{lemma}
Let $f:\mathbb{R}^d \to\mathbb{R}$ be a function with compact support such that
all partial derivatives up to order $k\in\mathbb{N}_0$ are continuous. Let
$\overline{\Omega}\subset \mathbb{R}^d$ be a compact set that contains the
support of $f$.
Then there exists a constant $C\in\mathbb{R}$
(independent of $n$ and $f$) such that for each $n\in\mathbb{N}$ a
polynomial $\pi\in P_n$ can be found such that for all $\alpha\in\mathbb{N}_0^d$
with $|\alpha|\leq \min\{k, n\}$
\begin{alignat*}{1}
&
\left\| \frac{\partial^{|\alpha|}  (f(\vx)-\pi(\vx))}{\partial
x_1^{\alpha_1}\dots\partial x_d^{\alpha_d}} \right\|_{C(\overline{\Omega})}
\leq \frac{C}{n^{k-|\alpha|}} \max_{\beta\in\mathbb{N}_0^d,
|\beta| = k} \omega_1\left( 
\frac{\partial^{k} f}{\partial x_1^{\beta_1}\dots\partial x_d^{\beta_d}},
\frac{1}{n}\right)_{\infty,\Omega}.
\end{alignat*}
\end{lemma}
Similar to the proof of Theorem \ref{theorem1},
we combine this cited result with Lemma \ref{LaIto} to obtain (cf.~\cite{Xie11})
\begin{theorem}[Synchronous Sup-Norm Approximation\label{theoremync}]
Let $\sigma :\mathbb{R}\to\mathbb{R}$ be arbitrarily often differentiable without being a polynomial.
For each function $f:\mathbb{R}^d \to\mathbb{R}$ with compact support and
continuous partial derivatives up to order $k\in\mathbb{N}_0$ and each compact
set $\overline{\Omega}\subset \mathbb{R}^d$ containing the
support of $f$ following estimate holds true:
For each $n\in\mathbb{N}$, $n\geq 4^d$, there exists a constant $C\in\mathbb{R}$
(independent of $n$ and $f$) such that for all $\alpha\in\mathbb{N}_0^d$ with $|\alpha|\leq 
\min\{k, \root d \of n -1\}$
\begin{alignat}{1}
\inf&\left\{
 \left\| \frac{\partial^{|\alpha|} (f(\vx)-g(\vx))}{\partial
 x_1^{\alpha_1}\dots\partial x_d^{\alpha_d}} 
 \right\|_{C(\overline{\Omega})} :
 g\in{\cal M}_n\right\} \nonumber\\
&\leq \frac{C}{n^{\frac{k-|\alpha|}{d}}}  \max_{\beta\in\mathbb{N}_0^d, |\beta|
= k} \omega_1\left(\frac{\partial^{k} f}{\partial x_1^{\beta_1}\dots\partial
x_d^{\beta_d}}, \frac{1}{\root d \of
n}\right)_{\infty,\Omega}.\label{directsync}
\end{alignat}
\end{theorem}

Requirements of Theorems \ref{theorem1} and \ref{theoremync} are not fulfilled
for activation functions that are of type 
\begin{equation}
 \sigma(x) = \left\{
\begin{array}{cl} 0,& x<0\\
x^k ,& x\geq 0\end{array}\right.\label{defrelu}
\end{equation}
for $k\in\mathbb{N}$. The often used ReLU function is obtained for $k=1$.
Corollary 6.11 in \cite[p.~178]{pinkus_1999} is an $L^2$-norm Jackson estimate
for this class of functions. To work with this estimate, we need to introduce
Sobolev spaces. 

Let $W_p^{r}(\Omega)$, $1\leq p< \infty$, be the $L^p$-Sobolev space of
$r$-times partially differentiable functions (in the weak sense) on
$\Omega\subset\mathbb{R}^d$ with semi-norms 
$$ |f|_{W_p^{r}(\Omega)} = \sum_{\alpha\in\mathbb{N}_0^d : |\alpha| = r} \left\|
 \frac{\partial^r f}{\partial x_1^{\alpha_1}\dots \partial x_d^{\alpha_d}}
 \right\|_{L^p(\Omega)}
$$
and norm
$ \|f\|_{W_p^{r}(\Omega)} = \sum_{k=0}^r |f|_{W_p^{r}(\Omega)}$.
For $r$-times continuously differentiable functions $f$ (case $p=\infty$) or
functions $f\in W_p^{r}(\Omega)$ on LG-domains $\Omega$, $1\leq p<\infty$, the
estimate
\begin{equation}
 \omega_r(f,\delta)_{p,\Omega} \leq C_r \delta^r
 \sum_{\alpha\in\mathbb{N}_0^d : |\alpha|=r} \left\|
 \frac{\partial^r f}{\partial x_1^{\alpha_1}\dots \partial x_d^{\alpha_d}}
 \right\|_{X^p(\Omega)}\label{gegenAbl}
\end{equation}
holds true, see \cite{JohnenScherer}.

According to the Jackson estimate for activation
functions (\ref{defrelu}) in \cite[p.~178]{pinkus_1999}, let $d\geq 2$ and
$\Omega\subset\mathbb{R}^d$ be the $d$-dimensional unit ball.
Then there exists a constant $C>0$ such that for all $f\in
W_2^{r}(\Omega)$ with $\|f\|_{W_2^{r}(\Omega)}\leq 1$ and $r\in\mathbb{N}$ with
$r<k+1+\frac{d-1}{2}$ ($k$ being the exponent in (\ref{defrelu}))
$$ E({\cal M}_n, f)_{2,\Omega} \leq C
\frac{1}{n^\frac{r}{d}}. $$
Thus, for all $f\in W_2^{r}(\Omega)$ without restriction
$\|f\|_{W_2^{r}(\Omega)}\leq 1$ there holds true
$$E({\cal M}_n, f)_{2,\Omega} \leq C \|f\|_{W_2^{r}(\Omega)}
n^{-\frac{r}{d}}.$$
Due to \cite[p.~75]{Adams}, a constant $C_1$ exists independently of $f$ such
that 
$ \|f\|_{W_2^{r}(\Omega)} \leq C_1[\|f\|_{L^2(\Omega)} +
|f|_{W_2^{r}(\Omega)}]$. Together we obtain
\begin{equation}
E({\cal M}_n, f)_{2,\Omega} \leq C
\left[\|f\|_{L^2(\Omega)}+|f|_{W_2^{r}(\Omega)}\right]
\frac{1}{n^\frac{r}{d}}.\label{jacksoncut}
\end{equation}
This estimate can be extended to moduli of smoothness
using K-functional techniques. To this end, we introduce some definitions that
will also be needed in the next section for discussing sharpness.
A functional $T$ on a normed space $X$, i.e., $T$
maps $X$ into $\mathbb{R}$, is non-negative-valued, sub-linear, and bounded, iff
for all $f,g\in X,\, c\in\mathbb{R}$
\begin{alignat*}{1}
 &T(f) \geq 0,\\
 &T(f+g) \leq T(f) + T(g),\\
 &T(c f) = |c|T(f),\\
 &\| T\|_{X^\sim} := \sup\{ T(f) : f\in X,\, \|f\|_X\leq 1\} < \infty.
\end{alignat*}
The set $X^\sim$ consists of all non-negative-valued,
sub-linear, bounded functionals $T$ on $X$.

Since we deal with non-linear approximation, error functionals will not be
sub-linear. Instead we discuss remainders
$(E_{n})_{n=1}^\infty$, $E_n : X\to [0,\infty)$ that fulfill following
conditions for $m\in\mathbb{N}$, $f, f_1, f_2,\dots, f_m\in X$, and constants
$c\in\mathbb{R}$:
\begin{alignat}{1}
&E_{m\cdot n}\left(\sum_{k=1}^m f_k\right) \leq \sum_{k=1}^m
E_n(f_k),\label{mono0}\\
&E_n(cf) = |c| E_n(f),\label{mono0b}\\
&E_n(f) \leq D_n \|f\|_X,\label{semilin}\\
&E_n(f) \geq E_{n+1}(f).\label{mono}
\end{alignat}
Constant $D_n$ is independent of $f$.
For $E_n(f):=E({\cal M}_n, f)_{p,\Omega}$ these conditions are fulfilled.

\begin{lemma}[K-functional]
Let functionals
$(E_{n})_{n=1}^\infty$, $E_n : X\to [0,\infty)$ fulfill
(\ref{mono0}) and (\ref{mono}).
The functionals should also fulfill not only (\ref{semilin}) but a
stability inequality:
Let constant $D_n$ in (\ref{semilin}) be independent of $n$, i.e., 
\begin{equation}
E_n(f) \leq D_0 \|f\|_{X}\label{stability}
\end{equation}
for a constant $D_0>0$ and all $n\in\mathbb{N}$. 
Also, a Jackson-type inequality ($0<\varphi(n)\leq 1$)
\begin{equation}
E_n(g) \leq D_1 \varphi(n) [\|g\|_X + |g|_U],\label{jackson}
\end{equation}
$D_1>0$,
is required that holds
for all functions $g$ in a subspace $U\subset X$ with semi-norm $|\cdot|_U$. 
For $n\geq 2$ and a constant $D_2>0$, the sequence $(\varphi(n))_{n=1}^\infty$
has to fulfill
\begin{equation}
\varphi\left(\left\lfloor
\frac{n}{2}\right\rfloor\right) \leq D_2 \varphi(n).\label{phizus}
\end{equation} 
Via the Peetre
K-functional
$$K\left(\delta, f, X,U\right) := \inf \{ \|f-g\|_X + \delta|g|_U : g\in U\}$$
one can estimate 
$$ E_n(f) \leq  C\left[ K\left(\varphi(n) , f, X, U\right) + \varphi(n)
\|f\|_X\right]
$$
for $n\geq 2$ with a constant $C$ that is independent of $f$ and $n$.
\end{lemma}
\begin{proof}
Let $g\in U$. Then
\begin{eqnarray*}
E_{2n}(f) &=& E_{2n}(f-g+g) \stackrel{\text{(\ref{mono0})}}{\leq} 
E_n(f-g) +E_n(g)\\ 
& \stackrel{\text{(\ref{stability}), (\ref{jackson})}}{\leq}&
D_0 \|f-g\|_{X} +  D_1 \varphi(n)  [\|g\|_X + |g|_U]\\
&\leq& D_0 \|f-g\|_{X} +  D_1 \varphi(n)  [\|f\|_X+\|g-f\|_X + |g|_U]\\
&\leq& (D_0+D_1) \|f-g\|_{X} +  D_1 \varphi(n) |g|_U + D_1 \varphi(n) \|f\|_X,
\end{eqnarray*}
thus for $n\geq 2$:
\begin{eqnarray*}
\lefteqn{E_{n}(f) \stackrel{\text{(\ref{mono})}}{\leq}  E_{2\lfloor
\frac{n}{2}\rfloor}(f)}\\ 
&\leq& (D_0+D_1)\left[ \inf \left\{ \|f-g\|_{X}+ 
\varphi\left(\left\lfloor \frac{n}{2}\right\rfloor\right) |g|_U : g\in U
\right\} + \varphi\left(\left\lfloor
\frac{n}{2}\right\rfloor\right) \|f\|_X\right]\\
&\stackrel{\text{(\ref{phizus})}}{\leq}& (D_0+D_1)\left[ \inf
\left\{ \|f-g\|_{X}\! +\! D_2 \varphi(n) |g|_U : g\in U \right\} + 
D_2\varphi(n) \|f\|_X\right] \\
& \leq& (D_0+D_1)\max\{1, D_2\}\left[ K\left(\varphi(n), f, X, U\right) + 
\varphi(n) \|f\|_X\right].
\end{eqnarray*}
\end{proof}
We apply the lemma to (\ref{jacksoncut}) with $X=L^2(\Omega)$,
$U=W_2^r(\Omega)$, $\varphi(n)=n^{-\frac{r}{d}}$. Error functional $ E({\cal
M}_n, f)_{2,\Omega}$ fulfills all prerequisites. In connection with the
equivalence between K-functionals and moduli of smoothness
\cite[p.~120]{JohnenScherer} we get 

\begin{theorem}[Piecewise Polynomial Functions]
Let $d\geq 2$, $\Omega\subset\mathbb{R}^d$ be the $d$-dimensional unit ball and
$\sigma$ a piecewise polynomial activation function of type (\ref{defrelu}).
Constants $C_1, C_2\in\mathbb{R}$ exist such that for each $f\in L^2(\Omega)$,
$n\geq 2$, $r<k+1+\frac{d-1}{2}$: 
\begin{alignat}{1}
 E({\cal M}_n, f)_{2,\Omega} &\leq C_1
 \left[K\left(\frac{1}{n^\frac{r}{d}}, f, L^2(\Omega), W_2^r(\Omega) \right) +
 \frac{1}{n^\frac{r}{d}} \|f\|_{L^2(\Omega)}\right]\nonumber\\
&\leq C_2 \left[\omega_r\left(f, \frac{1}{\root d \of n}
\right)_{2,\Omega}+ \frac{1}{n^\frac{r}{d}}
\|f\|_{L^2(\Omega)}\right].\label{cutest}
\end{alignat}
\end{theorem}
The saturation order of the modulus is $n^{-\frac{r}{d}}$, so term
$n^{-\frac{r}{d}} \|f\|_{L^2(\Omega)}$ is
only technical. The estimate also holds for ReLU ($k=1$) with only one ($d=1$) input node for
$r=2$, see \cite{Goebbels20}. It can be extended to the cut activation function
because cut can be written as a difference of ReLU and translated ReLU.

\section{Sharpness due to Counterexamples}\label{secount}

A coarse lower estimate can be obtained
for all integrable activation functions in the $L^2$-norm based on an estimate
for ridge functions in \cite{Maiorov96}. However, the general setting leads to
an exponent $\frac{r}{d-1}$ instead of $\frac{r}{d}$. 

The space of all
measurable, real-valued functions that are integrable on every compact subset of
$\mathbb{R}$ is denoted by $L(\mathbb{R})$.

\begin{lemma}
Let $\sigma$ be an
arbitrary activation function in $L(\mathbb{R})$
and $r\in\mathbb{N}$, $d\geq 2$. Let $\Omega$ be the $d$-dimensional unit ball.

Then there exists a sequence $(f_n)_{n=1}^\infty$, $f_n\in
W_2^{r}(\Omega)$, with $\|f_n\|_{W_2^{r}(\Omega)}\leq C_0$, and a constant
$c>0$ such that (cf.~Theorem \ref{theorem1}) 
$$ \omega_r\left(f_n, \frac{1}{\root d \of n}\right)_{2,\Omega} =
{O}\left(\frac{1}{n^{\frac{r}{d}}}\right)
\text{~and~~} 
E({\cal M}_n,
f_n)_{2, \Omega} \geq  \frac{c}{n^{\frac{r}{d-1}}}.
$$
\end{lemma}
\begin{proof}
This is a direct corollary of Theorem 1 in \cite{Maiorov96}:
For $A\subset\mathbb{R}^d$ with cardinality $|A|$ let $R(A)$ be the linear space
that is spanned by all functions $h(\vw\cdot\vx)$, $h\in L(\mathbb{R})$, $\vw\in A$. Thus in
contrast to one activation function, different nearly arbitrary functions $h$
are allowed to be used with different vectors $\vw$ in linear combinations. Let ${\cal
R}_n :=\bigcup_{A\subset\mathbb{R}^d : |A|\leq n} R(A)$ be the space of
functions that can be represented as $\sum_{k=1}^n a_k h_k(\vw_k\cdot\vx)$,
$a_k\in\mathbb{R}$, $h_k\in L(\mathbb{R})$, $\vw_k\in \mathbb{R}^d$.
Then for all activation functions 
$\sigma\in L(\mathbb{R})$ one has $h_k(x):=\sigma(x+c_k)\in L(\mathbb{R})$
for $c_k\in\mathbb{R}$, i.e.~${\cal M}_n\subset {\cal R}_n$. According to
\cite{Maiorov96}, for $d\geq 2$ there exist constants $0<c\leq C$ independently of $n$ such that
\begin{equation*}
 \frac{c}{n^{\frac{r}{d-1}}} \leq \sup_{f\in W_2^{r}(\Omega),
 \|f\|_{W_2^{r}(\Omega)}\leq C_0} \inf_{h\in{\cal R}_n} \|f-h\|_{L^2(\Omega)}
 \leq
\frac{C}{n^{\frac{r}{d-1}}}.
\end{equation*} 
From this condition, we obtain functions $f_n\in 
W_2^{r}(\Omega)$, $\|f_n\|_{W_2^{r}(\Omega)}\leq C_0$, such that 
\begin{alignat*}{1}
 \frac{1}{2}\frac{c}{n^{\frac{r}{d-1}}} &\leq 
\inf_{h\in{\cal R}_n} \|f_n -h\|_{L^2(\Omega)}
\leq \inf_{h\in{\cal
M}_n} \|f_n -h\|_{L^2(\Omega)} = E({\cal M}_n,
f_n)_{2, \Omega},
\end{alignat*}
and (see (\ref{gegenAbl}))
\begin{alignat*}{1}
\omega_r\left(f_n, \frac{1}{\root d \of n}\right)_{2,\Omega}&\leq
C_1 \frac{1}{n^\frac{r}{d}}
\sum_{\alpha\in\mathbb{N}_0^d : |\alpha|=r}\left\|
 \frac{\partial^r f_n}{\partial x_1^{\alpha_1}\dots \partial x_d^{\alpha_d}}
 \right\|_{L^2(\Omega)}
 \leq C_2(r,d) \frac{1}{n^\frac{r}{d}}.
\end{alignat*}
\end{proof}

By considering properties of the activation function, better lower estimates are
possible. 
For the logistic activation function and activation functions
that are splines of fixed polynomial degree with finite number of
knots like (\ref{defrelu}), Maiorov and Meir showed that there exists a sequence
$(f_n)_{n=2}^\infty$, $f_n\in W_p^{r}(\Omega)$, $r\in \mathbb{N}$, with
$\|f_n\|_{W_p^{r}(\Omega)}$ uniformly bounded, and a constant $c>0$ (independent
of $n\geq 2$) such that (see \cite[Theorem 4 and Theorem 5, p.~99, Corollary 2,
p.~100]{Maiorov99})
\begin{equation} 
E({\cal M}_n,
f_n)_{p, \Omega} \geq  \frac{c}{(n\log_2(n))^{\frac{r}{d}}}\label{eqPink}
\end{equation}
for $1\leq p< \infty$ (and $L^\infty(\Omega)$, but we consider
$C(\overline{\Omega})$ due to the definition of moduli of smoothness).
Without explicitly saying so, the proof is based on a VC dimension argument
similar to the proof of Theorem \ref{thVC} that follows in this section. It uses
\cite[Lemma 7, p.~99]{Maiorov99}. The formula in line 4 on page 98 of
\cite{Maiorov99} shows that (by choosing parameter $m$ as in the proof of
\cite[Theorem 4]{Maiorov99}) one additionally has 
\begin{alignat}{1}
  \|f_n\|_{L^p(\Omega)} &\leq\frac{C}{(n\log_2(n))^{\frac{r}{d}}}
  =  \frac{C}{(n(1+\log_2(n)))^{\frac{r}{d}}} 
  \left[\frac{1+\log_2(n)}{\log_2(n)}\right]^{\frac{r}{d}}\nonumber\\
  &\leq 
  \frac{2^{\frac{r}{d}} C }{(n(1+\log_2(n)))^{\frac{r}{d}}}.\label{boundedfn}
\end{alignat}
This result was proved for $\Omega$ being the unit ball. But similar to
Theorem \ref{thVC} below, a grid is used that can also be adjusted to
$\Omega=(0,1)^d$.


We now apply a resonance principle from \cite{Goebbels20} that is a
straight-forward extension of a general theorem by Dickmeis, Nessel and van
Wickern, see \cite{DiNeWi2}. With this principle, we condense sequences
$(f_n)_{n=1}^\infty$ like the one in (\ref{eqPink}) to single counterexamples.

To measure convergence rates,
abstract moduli of smoothness $\omega$ are often used, see
\cite[p.~96ff]{Timan63}.
An abstract modulus of smoothness is a continuous,
increasing function $\omega : [0,\infty) \to [0,\infty)$ such that for
$\delta_1,\delta_2 > 0$
\begin{equation}
 0=\omega(0)<\omega(\delta_1)\leq\omega(\delta_1+\delta_2)
 \leq\omega(\delta_1)+\omega(\delta_2). \label{ABSMOD}
\end{equation}
Typically, Lipschitz classes are defined via $\omega(\delta):=\delta^\alpha$,
$0 < \alpha\leq 1$.

\begin{theorem}[Adapted Uniform Boundedness Principle, 
see \cite{Goebbels20}]\label{ZUBP}
Let $(E_{n})_{n=1}^\infty$ be a sequence of remainders that map elements of a
real Banach space $X$ to non-negative numbers, i.e.,
$$E_n : X\to
[0,\infty).$$
The sequence has to fulfill conditions
(\ref{mono0})--(\ref{mono}).
Also, a family of sub-linear bounded functionals $S_\delta
\in X^\sim$ for all $\delta>0$ is given. These functionals will represent
moduli of smoothness.
To express convergence rates, let
$$\mu: (0,\infty)\to  (0,\infty) \text{ and } 
\varphi: [1,\infty)\to  (0,\infty)$$ 
be strictly decreasing
with $\lim_{x\to\infty} \varphi(x)=0$. Since remainder functionals $E_n$ are not
required to be sub-linear, $\varphi$ also has to fulfill following condition.
For each $0<\lambda<1$ there has to be
a real number $X_0=X_0(\lambda)\geq \lambda^{-1}$ and
constant $C_\lambda>0$ such that for all $x>X_0$ there holds
\begin{equation}
\varphi(\lambda x)\leq
C_\lambda \varphi(x).\label{faktor}
\end{equation}

If test elements $h_n\in X$ and a number $n_0\in\mathbb{N}$ exist such that for
all $n\in\mathbb{N}$ with $n\geq n_0$ and for all $\delta>0$
\begin{alignat}{1}
 \| h_n\|_X &\leq C_1,
           \label{ZUBP1}\\
  S_\delta(h_n) &\leq C_2 \min \left\{ 1,\frac{\mu(\delta)}{\varphi(n)}
           \right\},
           \label{ZZUBP3}\\
  E_{4n}(h_n) &\geq c_{3} >0,\label{ZUBP3}
\end{alignat}
then for each abstract modulus of smoothness $\omega$ satisfying
(\ref{ABSMOD}) and 
\begin{equation}
  \lim_{\delta\to 0+}\frac{\omega(\delta)}{\delta}=\infty\label{odd}
\end{equation}
a counterexample $f_\omega\in
X$ exists such that 
\begin{alignat*}{1}
 S_\delta(f_\omega) &= {O}\left(\omega(\mu(\delta))\right)\text{ for } \delta\to
 0+
\end{alignat*}
and 
\begin{alignat*}{1}
& E_{n}(f_\omega) \not= o(\omega(\varphi(n)))\text{ for }  n\to\infty\text{,
i.e., }
\limsup_{n\to\infty} \frac{E_{n}(f_\omega)}{\omega(\varphi(n))} > 0.
\end{alignat*}
\end{theorem}

When dealing with the sup-norm, one can generally apply the resonance theorem
in connection with known VC dimensions of indicator functions.
The general definition of VC dimension based on sets is as
follows.

Let $X$ be a finite set and ${\cal A}\subset {\cal P}(X)$ a family of subsets of
$X$.
Set $S\subset X$ is said to be shattered by ${\cal A}$ iff each subset $B\subset
S$ can be represented as $B=S\cap A$ for a family member $A\in {\cal A}$. Thus,
the set $\{S\cap A : A\in {\cal A}\}$ has $2^{|S|}$ elements, $|S|$
denoting the cardinality $S$.
\begin{alignat*}{1}
\operatorname{VC-dim}({\cal A}) :=  \sup \{&k\in\mathbb{N} : \exists S\subset X
\text{ with cardinality }\\
& |S|=k \text{ such that } S \text{ is shattered
by } {\cal A}\}
\end{alignat*}
is called the VC dimension of ${\cal A}$.

For our purpose, we discuss a (non-linear) set $V$ of functions
$g:X\to\mathbb{R}$ on a set $X\subset\mathbb{R}^m$. Using
Heaviside-function $H:\mathbb{R}\to \{0, 1\}$, 
\begin{alignat*}{1}
H(x) &:= \left\{ \begin{array}{cc}0,& x<0\\ 1,& x\geq
0,\end{array}\right.
\end{alignat*}
let
\begin{alignat*}{1}
 {\cal A} := \{ A\subset X : \exists g\in V : &\, (\forall x\in A:
H(g(x))=1)\, \und\,
 (\forall x\in X\setminus A: H(g(x))=0) \}.
\end{alignat*}
Then one typically defines
$\operatorname{VC-dim}(V):=\operatorname{VC-dim}({\cal A})$.
Thus, $k:=\operatorname{VC-dim}(V)$ is the largest cardinality of a
subset 
$S=\{x_1, \dots, x_k\}\subset X$ 
such that for each sign sequence
$s_1,\dots, s_k\in\{-1, 1\}$ a function $g\in V$ can be found that fulfills (cf.~\cite{Bartlett1996})
$$H(g(x_i)) = H(s_i),\quad 1\leq i\leq k.
$$

\begin{theorem}[Sharpness due to VC Dimension]\label{thVC}
Let 
$(V_n)_{n=1}^\infty$ be a sequence of (non-linear) function spaces $V_n$
of bounded real-valued functions on $[0, 1]^d$ such that 
\begin{equation}
E_n(f):=\inf \{ \| f-g\|_{C([0, 1]^d)} : g\in
V_n\}\label{defen}
\end{equation}
fulfills conditions (\ref{mono0})--(\ref{mono}) on Banach space $C([0, 1]^d)$.
An equidistant grid $X_n\subset [0, 1]^d$ with a step size $\frac{1}{\tau(n)}$, 
$\tau :\mathbb{N} \to \mathbb{N}$, is given via
$$X_{n}:=\left\{
\frac{j}{\tau(n)} :
j\in\{0,1,\dots,\tau(n)\}\right\}\times\dots\times\left\{
\frac{j}{\tau(n)} :
j\in\{0,1,\dots,\tau(n)\}\right\}.$$ 
Let
$$
 V_{n,\tau(n)}  := \{  h:X_n\to\mathbb{R} : \text{ a function } g\in
 V_n \text{ exists with } h(\vx)=g(\vx)
\text{ for all } \vx\in X_n \}
$$
be the set of functions that are generated by restricting functions of $V_n$
to this grid.
As in Theorem \ref{ZUBP}, convergence rates are expressed via a function $\varphi(x)$ 
that fulfills the requirements of Theorem \ref{ZUBP} including condition
(\ref{faktor}). 
Let VC dimension of $V_{n,\tau(n)}$ and function values of $\tau$ and $\varphi$
be coupled via inequalities
\begin{eqnarray}
\operatorname{VC-dim}(V_{n,\tau(n)}) &<& [\tau(n)]^d,\label{VCcond}\\
\tau(4n) &\leq& \frac{C}{\varphi(n)}, \label{alphacond}
\end{eqnarray}
for all $n\geq n_0\in\mathbb{N}$ with a constant $C>0$ that is independent of
$n$.

Then, for $r\in\mathbb{N}$ 
and each abstract modulus of smoothness $\omega$ satisfying 
(\ref{ABSMOD}) and (\ref{odd}), there exists a counterexample $f_{\omega}\in
C([0,1]^d)$ such that for $\delta\to 0+$ and $n\to\infty$
$$ \omega_r(f_{\omega}, \delta)_{\infty,(0,1)^d} =
{O}\left(\omega(\delta^r)\right)
\text{ and }
E_n(f_{\omega}) \neq
o\left(\omega\left([\varphi(n)]^r\right)\right).
$$
\end{theorem} 
\begin{proof}
Condition (\ref{VCcond}) implies for $4n\geq n_0$ that a sequence of signs 
$s_{\vz}\in\{-1, 1\}$ for points $\vz\in
X_{4n}$ exists such that no function in $V_{4n}$ can reproduce the sign of the
sequence in each point of $X_{4n}$, i.e.,
for each $g\in V_{4n}$ there exists a point $\vz_0
\in X_{4n}$ such that 
$$H(g(\vz_0)) \neq H(s_{\vz_0}).$$

Based on this sign sequence, we construct an arbitrarily often
partially differentiable resonance function $h_n$ such that its function values
equal the signs on the grid $X_{4n}$.
To this end, we use the arbitrarily often differentiable function
\begin{equation*}
 h(x) := \left\{ \begin{array}{ll}
\exp\left(1-\frac{1}{1-x^2}\right) & \text{for } |x|<1,\\
0 & \text{for } |x|\geq 1,
\end{array}\right.
\end{equation*}
with properties $h(0)=1$ and
$\|h\|_{B(\mathbb{R})}=1$.
Based on $h$, we define $h_n$:
$$ h_n(\vx):= \sum_{\vz\in
X_{4n}} s_{\vz} \cdot
\prod_{k=1}^d h\left( 2\cdot\tau(4n)\cdot\left(\vx_k-\vz_k\right)\right).$$ 
Scaling factors $2\cdot\tau(4n)$ are chosen such that supports of
summands only intersect at their borders. Therefore,
$\|h_n\|_{C([0,1]^d)}\leq 1$ and $h_n(\vz)=s_{\vz}$ for all $\vz\in
X_{4n}$.

All partial derivatives of order up to $r$ are in
$O([\varphi(n)]^{-r})$ because of (\ref{alphacond}).
Additionally to $h_n$, we choose parameters in Theorem \ref{ZUBP} as
follows:
\begin{alignat*}{1}
&X=C([0, 1]^d)\text{, }
S_\delta(f) := \omega_r(f, \delta)_{\infty,(0,1)^d}\text{, }
\mu(\delta):=\delta^r\text{,}
\end{alignat*}
and $E_n(f)$ as in (\ref{defen}). 
We do not directly use $\varphi(x)$ with Theorem \ref{ZUBP}. Instead, 
function $[\varphi(x)]^r$ fulfills the requirements of the function
also called $\varphi(x)$ in Theorem \ref{ZUBP}.

Requirements (\ref{ZUBP1}) and (\ref{ZZUBP3}) can be easily shown 
due to the sup-norms of $h_n$ and its partial derivatives, cf.\
(\ref{gegenAbl}).

Resonance condition (\ref{ZUBP3}) is fulfilled due to the definition of $h_n$:
For each $g\in V_{4n}$ there exits at least one point $\vz_0\in
X_{4n}$ such that 
$$H(g(\vz_0))\neq H(s_{\vz_0})=H(h_n(\vz_0)).$$
Function $h_n$ is defined to fulfill $|h_n(\vz_0)|=|s_{\vz_0}|=1$. Thus,
$$\|h_n-g\|_{C([0,1]^d)} \geq |h_n(\vz_0)-g(\vz_0)| \geq 1,$$ 
and $E_{4n} h_n
\geq 1$.

All preliminaries of Theorem \ref{ZUBP} are fulfilled such that counterexamples
exist as stated.
\end{proof}

\begin{theorem}[Sharpness for Logistic
Function Approximation in Sup-Norm]\label{corsharp} Let $\sigma$ be the
logistic function and $r\in\mathbb{N}$.
For each abstract modulus of smoothness $\omega$ satisfying 
(\ref{ABSMOD}) and (\ref{odd}),
a counterexample
$f_{\omega}\in C([0,1]^d)$ exists such that for $\delta\to 0+$
$$ \omega_r(f_{\omega}, \delta)_{\infty,(0,1)^d} =
{O}\left(\omega(\delta^{r})\right)$$ and 
for $n\to \infty$
$$E({\cal M}_n,
f_{\omega})_{\infty, (0,1)^d} \neq o\left(\omega\left(\frac{1}{(n
[1+\log_2(n)])^{\frac{r}{d}}}\right)\right).
$$
\end{theorem} 

For univariate approximation, i.e., $d=1$, the theorem is proved in
\cite{Goebbels20}. This proof can be generalized as follows.

\begin{proof}
Let $D\in\mathbb{N}$. In \cite{Bartlett1996}, an upper bound for the VC
dimension of function spaces
\begin{eqnarray*}
 \Delta_n &:=&\big\{g : \{-D,-D+1,\dots, D\}^d \to\mathbb{R} :\\
&&\qquad g(x)=a_0+ \sum_{k=1}^{n} a_k \sigma(\vw_k \cdot\vx+c_k)\text{, }
a_0,a_k,c_k\in\mathbb{R}, \vw_k\in\mathbb{R}^d\big\}
\end{eqnarray*}
is derived.
Functions are defined on a discrete set with $(2D+1)^d$ points. Please note
that the constant function $a_0$ is not consistent with the
definition of ${\cal M}_n$. It provides an additional degree of freedom.

We apply Theorem 2 in \cite{Bartlett1996}: There exists $n^\ast\in\mathbb{N}$ 
such that for all $n\geq n^\ast$ the VC
dimension of $\Delta_n$ is upper bounded by
$$2 \cdot (nd +2n+1) \cdot
\log_2 (24e(nd+2n+1)D),$$
i.e., there exists an $n_0\geq \max\{2, n^\ast\}$, $n_0\in\mathbb{N}$, and a
constant $C_d>0$, dependent on $d$, such that for all $n\geq n_0$ 
$$\operatorname{VC-dim}(\Delta_n) \leq C_d n[\log_2(n) +
\log_2(D)].
$$

Let constant $E>1$ be chosen such that 
\begin{alignat}{1}
&\frac{1+\log_2(E)}{E} < \frac{1}{4C_d}\text{, i.e, }
 4C_d [1+\log_2(E)] < E.
\label{abE}
\end{alignat}
This is possible because 
$$\lim_{E\to\infty} \frac{1+\log_2(E)}{E}=0.$$
Now we choose a suitable value of $D=D(n)$ such that the VC dimension of
$\Delta_n$ is less than $[D(n)]^d$. To this end, let
$$D=D(n):= \left\lfloor \root d
\of {En(1+\log_2(n))}\right\rfloor.$$ 
Then we get for $n\geq n_0$ with
(\ref{abE}):
\begin{eqnarray*}
\operatorname{VC-dim}(\Delta_n) &\leq&
C_d n[\log_2(n) + \log_2(\root d \of {E n(1+\log_2(n))} )]\\
&=& C_d n[\log_2(n) + d^{-1} \log_2(E n(1+\log_2(n)) )]\\
&\leq&
C_d n[2\log_2(n) + \log_2(E) +\log_2(2\log_2(n)) )]\\
&\leq&
C_d n[3\log_2(n) + \log_2(E) + 1  )]
\leq 4C_d n\log_2(n)[1+\log_2(E)]\\
&<& E n\log_2(n) \leq \left\lfloor
\root d \of { E n(1+\log_2(n))} \right\rfloor^d = [D(n)]^d.
\end{eqnarray*}

One can map interval $[-D, D]^d$ to $[0,
1]^d$ with an affine transform. By also omitting constant
$a_0$, we estimate the VC dimension of $V_{n,\tau(n)}$ with
parameters $V_n:={\cal M}_n$ and $\tau(n):=2D(n):$ 
$$
\operatorname{VC-dim}(V_{n,\tau(n)}) < [D(n)]^d < [2D(n)]^d = [\tau(n)]^d.
$$
Thus, (\ref{VCcond}) is fulfilled.
Conditions (\ref{mono0})--(\ref{mono}) are chosen such that they fit with error functionals 
$$E_n:=E({\cal
M}_n, \cdot)_{\infty, (0,1)^d}.$$

For strictly decreasing function
$$\varphi(x):=1/\root d \of {x[1+\log_2(x)]}$$ 
conditions
$\lim_{x\to\infty} \varphi(x) = 0$ and (\ref{faktor}) hold. Latter can be shown
for $x>X_0(\lambda):=\lambda^{-2}$ because
$\log_2(\lambda) > -\log_2(x)/2$ and
\begin{alignat*}{1}
\varphi(\lambda x) &=\frac{1}{\root d \of \lambda}
\frac{1}{\root d \of {x(1+\log_2(x)+\log_2(\lambda))}}
\leq
\frac{1}{\root d \of \lambda} \frac{1}{\root d \of {x(1+\frac{1}{2}\log_2(x))}}
< \root d \of {\frac{2}{\lambda}} \varphi(x).
\end{alignat*}

Finally, (\ref{alphacond}) follows from  
\begin{alignat*}{1}
\tau(4n) &= 2D(4n) \leq
2\root d \of {E4n(1+\log_2(4n))}
< 
\frac{2\root d\of {4E(1+\log_2(4))}}{\varphi(n)}
=\frac{2\root d
\of {12E}}{\varphi(n)}.
\end{alignat*}
Thus, Theorem \ref{thVC} can be applied to obtain the counterexample.
\end{proof}

The theorem can also be proved based on the sequence $(f_n)_{n=1}^\infty$
from \cite{Maiorov99} with properties (\ref{eqPink}) and (\ref{boundedfn}).
We use this sequence to obtain the sharpness in $L^p$ norms for approximation
with piecewise polynomial activation functions as well as with the
logistic function.

Theorem 1 in \cite{Ratsaby99} provides a general means to obtain such bounded
sequences in Sobolev spaces for which approximation by functions in ${\cal M}_n$ is lower
bounded with respect to pseudo-dimension. 

We condense sequence
$(f_n)_{n=1}^\infty$ to a single counterexample with the next theorem.

\begin{theorem}[Sharpness with $L^p$-Norms]\label{lpcorsharp} 
Let $\sigma$ be either the
logistic function or a piecewise polynomial activation function of type (\ref{defrelu})
and $r\in\mathbb{N}$. Let $\Omega$ be the $d$-dimensional unit ball,
$d\in\mathbb{N}$, $1\leq p<\infty$. 
For each abstract modulus of smoothness $\omega$ satisfying 
(\ref{ABSMOD}) and (\ref{odd}),
a counterexample
$f_{\omega}\in L^p(\Omega)$ exists such that for $\delta\to 0+$
$$ \omega_r(f_{\omega}, \delta)_{p, \Omega} =
{O}\left(\omega(\delta^{r})\right)$$ and 
for $n\to \infty$
$$E({\cal M}_n,
f_{\omega})_{p, \Omega} \neq o\left(\omega\left(\frac{1}{(n
[1+\log_2(n)])^{\frac{r}{d}}}\right)\right).
$$
\end{theorem} 

\begin{proof}
We apply Theorem \ref{ZUBP} with following parameters for $n\geq 2$:
\begin{alignat*}{1}
&E_n(f):=E({\cal M}_n,
f)_{p, \Omega}, \, X=L^p(\Omega),\, S_\delta(f)=\omega_r(f,
\delta)_{p, \Omega},\\
& \varphi(x)=\frac{1}{[x(1+\log_2(x))]^\frac{r}{d}}, \quad \mu(\delta)=\delta^r. 
\end{alignat*}
Function $\varphi(x)$ satisfies
the prerequisites of Theorem \ref{ZUBP} similarly to the proof of Theorem
\ref{corsharp}.
Also, conditions
(\ref{mono0})--(\ref{mono}) hold true for $E_n$. 
For $r\in\mathbb{N}$, we use the sequence
$(f_n)_{n=2}^\infty$ of (\ref{eqPink}) to define resonance elements $$ h_n :=
\frac{1}{\varphi(4n)} \cdot f_{4n} $$
such that functions $h_n$ are uniformly bounded in $L^p(\Omega)$ due to
(\ref{boundedfn}) in $L^p(\Omega)$.
Thus, (\ref{ZUBP1}) is fulfilled.
From (\ref{eqPink}) we obtain resonance condition (\ref{ZUBP3})
\begin{alignat*}{1}
 E({\cal M}_{4n},
h_n)_{p, \Omega} &\geq 
[4n(1+\log_2(4n))]^\frac{r}{d} \cdot
\frac{c}{(4n(\log_2(4n))^{\frac{r}{d}}}
 \geq c > 0.
\end{alignat*}
Since $(\|f_n\|_{W_p^{r}(\Omega)})_{n=2}^\infty$ is bounded,
estimate (\ref{gegenAbl}) yields (\ref{ZZUBP3}):
\begin{alignat*}{1}
\omega_r(h_n, \delta)_{p, \Omega} &= \frac{1}{\varphi(4n)} \omega_r(f_{4n},
\delta)_{p, \Omega}
 \leq C \frac{\delta^r}{\varphi(4n)} 
\leq 12^\frac{r}{d} C \frac{\mu(\delta)}{\varphi(n)}. 
\end{alignat*}
Thus, all prerequisites of Theorem \ref{ZUBP} are fulfilled such that
counterexamples exist as stated.
\end{proof}

With respect to error bound (\ref{directsync}) for
synchronous approximation, counterexamples can be obtained due to the following
observation for $\alpha\in\mathbb{N}_0^d$, $|\alpha|=k$:
\begin{alignat*}{1}
\inf&\left\{
 \left\| \frac{\partial^{|\alpha|} (f(\vx)-g(\vx))}{\partial
 x_1^{\alpha_1}\dots\partial x_d^{\alpha_d}} 
 \right\|_{X^p(\Omega)}\!\!\!\! :
 g\in{\cal M}_{n,\sigma}\right\}
\geq E\left({\cal M}_{n,\sigma^{(k)}},
\frac{\partial^{|\alpha|} f(\vx)}{\partial
 x_1^{\alpha_1}\dots\partial x_d^{\alpha_d}} \right)_{p, \Omega}. 
\end{alignat*}
In the univariate case $d=1$, a counterexample for approximation with
$\sigma^{(k)}$ can be integrated to become a counterexample that shows
sharpness of (\ref{directsync}). For example,
$\sigma(x)=\frac{1}{2}+\frac{1}{\pi} \arctan(x)$ is discussed in
\cite[Corollary 4.2]{Goebbels20}. The given proof shows that for each abstract
modulus of smoothness $\omega$ satisfying (\ref{ABSMOD}) and (\ref{odd}),
a continuous counterexample $f_\omega'$ exists such that $\omega_1(f_{\omega}', \delta)_{\infty, \Omega} =
{O}\left(\omega(\delta)\right)$ and $E({\cal M}_{n,\sigma'}
f_\omega')_{\infty, \Omega} \neq o\left(\omega\left(\frac{1}{n}\right)\right)$.
Thus, one can choose $f_\omega(x):=\int_0^x f'_\omega(t)\,dt$.
In the multivariate case however, integration with respect to
one variable does not lead to sufficient smoothness with regard to other
variables.


\section{Conclusions}
By setting $\omega(\delta):=\delta^{\alpha\frac{d}{r}}$, 
we have shown the following for the logistic
function.
For each $0<\alpha<\frac{r}{d}$ condition (\ref{odd}) is fulfilled, and
according to Theorem \ref{corsharp} there exists a counterexample
$f_{\omega}\in C([0,1]^d)$ with $$ \omega_r(f_{\omega}, \delta)_{\infty,(0,1)^d} =
{O}\left(\delta^{d \alpha}\right)
\text{ and }
E({\cal M}_n,
f_{\omega})_{\infty, (0,1)^d}
=O\left(\frac{1}{n^\alpha}\right) $$
such that for all $\beta>\alpha$
$$E({\cal M}_n,
f_{\omega})_{\infty, (0,1)^d} \neq
O\left(\frac{1}{n^\beta}\right)
\text{ because } 
\frac{1}{n^\beta} = o\left(\frac{1}{(n[1+\log_2(n)])^\alpha}\right).$$
 
With Theorem \ref{lpcorsharp}, similar $L^p$ estimates for the logistic
function and $L^2$ estimates for piecewise polynomial activation functions
(\ref{defrelu}) hold true, see direct $L^2$-norm estimate (\ref{cutest}).
With one input node ($d=1$), a lower estimate for piecewise polynomial
activation functions without the $\log$-factor can be proved easily, see
\cite{Goebbels20}. Thus, the bound in Theorem \ref{lpcorsharp} might be
improvable.

Future work can deal
with sharpness of error bound (\ref{directsync}) for
synchronous approximation in the multivariate case. By extending quantitative uniform boundedness principles
with multiple error functionals (cf.~\cite{Dickmeis}, \cite{Imhof1},
\cite{Imhof2}) to non-linear approximation (cf.~proof of Theorem \ref{ZUBP} in
\cite{Goebbels20}), one might be able to show simultaneous sharpness in
different (semi-) norms like this conjecture: Under the preliminaries
of Theorem 2, the following might hold true for the logistic activation
function:
For each abstract modulus of continuity $\omega$ fulfilling (25) there exists a $k$-times continuously
differentiable counterexample $f_\omega$ such that for each
$r\in\{0,\dots,k\}$ it simultaneously fulfills ($\delta\to 0+$,
$n\to\infty$)
\begin{alignat*}{1}
&\omega_1\left(\frac{\partial^{k} f_\omega}{\partial x_1^{\beta_1}\dots\partial
x_d^{\beta_d}},\delta\right)_{\infty,\Omega} =
O\left(\omega\left(\delta\right)\right)\text{ for all }
\beta\in\mathbb{N}_0^d,\, |\beta| = k,\\
&\max_{\alpha\in\mathbb{N}_0^d,\, |\alpha|=r}
\inf\left\{
 \left\| \frac{\partial^{|\alpha|} (f_\omega(\vx)-g(\vx))}{\partial
 x_1^{\alpha_1}\dots\partial x_d^{\alpha_d}} 
 \right\|_{C(\overline{\Omega})} :
 g\in{\cal M}_n\right\}\\ 
 &\qquad \neq
 o\left(\frac{1}{(n(1+\log_2(n))^{\frac{k-r}{d}}}\cdot\omega\left(\frac{1}{\root
 d \of {n(1+\log_2(n))}}\right)\right).
\end{alignat*}
\bibliographystyle{spmpsci}
\bibliography{neural}

\end{document}